\documentclass[11 pt]{article}

\usepackage{amsfonts,amsmath,amssymb,hhline,ifthen,amsthm,graphics,cite,latexsym, caption2}
\usepackage{latexsym}
\usepackage[T2A]{fontenc}
\usepackage[cp1251]{inputenc}
\usepackage[english]{babel}

\usepackage[matrix, arrow, curve]{xy}
\usepackage{graphicx}

\newcounter{q3}
\newcommand{\dsm}[3]{
{\if#20{\if#31{\frac{\partial #1}{\partial y}}\else
          {\frac{\partial^{#3} #1}{\partial y^{#3}}}
        \fi}\else
  {\if#30{\if#21{\frac{\partial #1}{\partial x}}\else
            {\frac{\partial^{#2} #1}{\partial x^{#2}}}
          \fi}\else
    {\setcounter{q3}{#2}\addtocounter{q3}{#3}
    \frac{\partial{\if{1}\arabic{q3}^{}\else^{ \arabic{q3} }\fi}#1}
    {{\if#20\else{\partial x{\if#21\else{^{#2}}\fi}}\fi}
     {\if#30\else{\partial y\if#31\else{^{#3}}\fi}\fi} }}
   \fi}
\fi} }

\newcommand{\dzm}[2]{
{\if#10{\if#21{\frac{\partial}{\partial\overline{\zeta}}}\else
          {\frac{\partial^{#2}}{\partial\overline{\zeta}^{#2}}}
        \fi}\else
  {\if#20{\if#11{\frac{\partial}{\partial\zeta}}\else
            {\frac{\partial^{#1}}{\partial\zeta^{#1}}}
          \fi}\else
    {\setcounter{q3}{#1}\addtocounter{q3}{#2}
    \frac{\partial{\if{1}\arabic{q3}^{}\else^{ \arabic{q3} }\fi}}
    {{\if#10\else{\partial\zeta{\if#21\else{^{#1}}\fi}}\fi}
     {\if#20\else{\partial\overline{\zeta}\if#21\else{^{#2}}\fi}\fi} }}
   \fi}
\fi} }

\newcounter{q2}

\newtheoremstyle{theor}% name
  {\medskipamount}%      Space above
  {\medskipamount}%      Space below
  {\itshape}% Body font
  {}%\parindent}% Indent amount (empty = no indent, \parindent = para indent)
  {\bfseries}
  {.}
  {.5em}
  {}

\newtheorem{definition}{Definition}[section]
\newtheorem{theorem}[definition]{Theorem}
\newtheorem{lemma}[definition]{Lemma}
\newtheorem{proposition}[definition]{Proposition}
\newtheorem{corollary}[definition]{Corollary}

\theoremstyle{definition}
\newtheorem{remark}[definition]{Remark}
\newtheorem{example}[definition]{Example}

\numberwithin{equation}{section}

\makeindex

\newtheoremstyle{remarks}% name
  {0mm}%      Space above
  {0mm}%      Space below
  {\itshape}% Body font
  {}%\parindent}% Indent amount (empty = no indent, \parindent = para indent)
  {\itshape}
  {.}
  {.5em}
  {}
\makeatletter \@addtoreset{equation}{section} \makeatother

\begin{document}

\subsection*{\center UNIFORM BOUNDEDNESS OF PRETANGENT SPACES AND LOCAL STRONG ONE-SIDE POROSITY}\begin{center}\textbf{Viktoriia Bilet and Oleksiy Dovgoshey} \end{center}
\parshape=5
1cm 11.5cm 1cm 11.5cm 1cm 11.5cm 1cm 11.5cm 1cm 11.5cm \noindent \small {\bf Abstract.}
 Let $(X,d,p)$ be a pointed metric space. A pretangent space to $X$ at $p$ is a metric space consisting of some equivalence classes of convergent to $p$ sequences $(x_n), x_n \in X,$ whose degree of convergence is comparable with a given scaling sequence $(r_n), r_n\downarrow 0.$ We say that $(r_n)$ is normal if there is $(x_n)$ such that $\mid d(x_n,p)-r_n\mid=o(r_n)$ for $n\to\infty.$ Let $\mathbf{\Omega_{p}^{X}(n)}$ be the set of  pretangent spaces to $X$ at $p$  with normal scaling sequences.  We prove that the spaces from $\mathbf{\Omega_{p}^{X}(n)}$ are uniformly bounded if and only if $\{d(x,p): x\in X\}$ is a so-called completely strongly porous set.

\parshape=2
1cm 11.5cm 1cm 11.5cm  \noindent \small {\bf Key words:} metric space, tangent space to metric space, boundedness, local strong one-side porosity.

 \bigskip
\textbf{ AMS 2010 Subject Classification: 54E35, 28A10, 28A05}

%%%%%%%%%%%%%%%%%%%%%%%%%%%%%%%%%%%%%%%%%Introduction%%%%%%%%%%%%%%%%%%%%%%%%%%%%%%%%%%%%%%%%%%%%

\large\section {Introduction } \hspace*{\parindent}The pretangent and tangent spaces to
a metric space $(X,d)$ at a given point $p\in X$ were defined in \cite{DM}(see
also \cite{MD}). For convenience we recall some results and terminology related to the pretangent spaces.

Let $(X,d,p)$ be a pointed metric space with a metric $d$ and a marked point $p.$ Fix a
sequence $\tilde{r}$ of positive real numbers $r_n$ tending to zero. In what follows
$\tilde{r}$ will be called a \emph{scaling sequence}. Let us denote by $\tilde{X}$
the set of all sequences of points from X.
\begin{definition}\label{D1.1} Two sequences $\tilde{x}=(x_n)_{n\in \mathbb N}$ and $\tilde{y}=(y_n)_{n\in \mathbb
N},$ $\tilde{x}, \tilde{y} \in \tilde{X}$ are mutually stable with respect to
$\tilde{r}=(r_n)_{n\in \mathbb N}$ if the finite limit
\begin{equation}\label{eq1.2}
\lim_{n\to\infty}\frac{d(x_n,y_n)}{r_n}:=\tilde{d}_{\tilde{r}}(\tilde{x},\tilde{y})=\tilde{d}(\tilde{x},\tilde{y})\end{equation} exists.\end{definition}
We shall say that a family $\tilde{F}\subseteq\tilde{X}$ is \emph{self-stable} (w.r.t.
$\tilde{r}$) if any two $\tilde{x}, \tilde{y} \in \tilde{F}$ are mutually stable. A
family $\tilde{F}\subseteq\tilde{X}$ is \emph{maximal self-stable} if $\tilde{F}$ is
self-stable and for an arbitrary $\tilde{z}\in \tilde{X}$ either $\tilde{z}\in\tilde{F}$
or there is $\tilde{x}\in\tilde{F}$ such that $\tilde{x}$ and $\tilde{z}$ are not
mutually stable.

A standard application of Zorn's lemma leads to the following
\begin{proposition}\label{Pr1.2}Let $(X,d,p)$ be a pointed metric space. Then for every scaling sequence $\tilde{r}=(r_n)_{n\in \mathbb N}$ there exists a maximal self-stable family $\tilde{X}_{p,\tilde{r}}$ such that $\tilde{p}:=(p,p,...)~\in~\tilde{X}_{p,\tilde{r}}.$
\end{proposition}

Note that the condition $\tilde{p}\in\tilde{X}_{p,\tilde{r}}$ implies the equality
$\mathop{\lim}\limits_{n\to\infty}d(x_n,p)=0 $ for every $\tilde{x}=(x_n)_{n\in \mathbb N}\in\tilde{X}_{p,\tilde{r}}.$

Consider a function $\tilde{d}:\tilde{X}_{p,\tilde{r}}\times\tilde{X}_{p,\tilde{r}}\rightarrow\mathbb R$ where
$\tilde{d}(\tilde{x},\tilde{y})=\tilde{d}_{\tilde{r}}(\tilde{x},\tilde{y})$ is defined
by \eqref{eq1.2}. Obviously, $\tilde{d}$ is symmetric and nonnegative. Moreover, the
triangle inequality for $d$ implies
$$\tilde{d}(\tilde{x},\tilde{y})\leq\tilde{d}(\tilde{x},\tilde{z})+\tilde{d}(\tilde{z},\tilde{y})$$
for all $\tilde{x},\tilde{y},\tilde{z}\in\tilde{X}_{p,\tilde{r}}.$ Hence $(\tilde{X}_{p,\tilde{r}},\tilde{d})$
is a pseudometric space.
\begin{definition}\label{D1.3} A pretangent space to the space X (at the point p w.r.t. $\tilde{r}$) is the metric identification of a pseudometric space
$(\tilde{X}_{p,\tilde{r}},\tilde{d}).$\end{definition}

Since the notion of pretangent space is basic for the paper, we remind this metric
identification construction.

Define a relation $\sim$ on $\tilde{X}_{p,\tilde{r}}$ as $\tilde x\sim \tilde y$ if and only if
$\tilde d(\tilde x, \tilde y)=0.$ Then $\sim$ is an equivalence relation. Let us denote
by $\Omega_{p,\tilde r}^{X}$ the set of equivalence classes in $\tilde{X}_{p,\tilde{r}}$ under the
equivalence relation $\sim.$ It follows from general properties of the pseudometric spaces
(see, for example, \cite{Kelley}) that if $\rho$ is defined on $\Omega_{p,\tilde
r}^{X}$ as \begin{equation} \label{eq1.4}\rho(\gamma,\beta):=\tilde d (\tilde x, \tilde
y)\end{equation}for $\tilde x\in \gamma$ and $\tilde y\in \beta,$ then $\rho$ is a
well-defined metric on $\Omega_{p,\tilde r}^{X}.$ By definition, the metric
identification of $(\tilde{X}_{p,\tilde{r}}, \tilde d)$ is the metric space $(\Omega_{p,\tilde
r}^{X}, \rho).$

It should be observed that $\Omega_{p,\tilde r}^{X}\ne \varnothing$ because the constant
sequence $\tilde p$ belongs to $\tilde{X}_{p,\tilde{r}}.$ Thus every pretangent space
$\Omega_{p, \tilde r}^{X}$ is a pointed metric space with the natural distinguished point
$\alpha=\pi (\tilde p),$ (see diagram~\eqref{eq1.5} below).

Let $(n_k)_{k\in\mathbb N}$ be an infinite, strictly increasing sequence of natural
numbers. Let us denote by $\tilde r'$ the subsequence $(r_{n_k})_{k\in \mathbb N}$ of
the scaling sequence $\tilde r=(r_n)_{n\in\mathbb N}$ and let $\tilde
x':=(x_{n_k})_{k\in\mathbb N}$ for every $\tilde x=(x_n)_{n\in\mathbb N}\in\tilde
X.$ It is clear that if $\tilde x$ and $\tilde y$ are mutually stable w.r.t. $\tilde r,$
then $\tilde x'$ and $\tilde y'$ are mutually stable w.r.t. $\tilde r'$ and
$\tilde d_{\tilde r}(\tilde x, \tilde y)=\tilde d_{\tilde r'}(\tilde x', \tilde
y').$ If $\tilde X_{p,\tilde r}$ is a maximal self-stable (w.r.t $\tilde
r$) family, then, by Zorn's Lemma, there exists a maximal self-stable (w.r.t $\tilde
r'$) family $\tilde X_{p,\tilde r'}$ such that $$\{\tilde x':\tilde x \in \tilde
X_{p,\tilde r}\}\subseteq \tilde X_{p,\tilde r'}.$$ Denote by $in_{\tilde r'}$ the map
from $\tilde X_{p,\tilde r}$ to $\tilde X_{p,\tilde r'}$ with $in_{\tilde r'}(\tilde
x)=\tilde x'$ for all $\tilde x\in\tilde X_{p,\tilde r}.$ It follows from \eqref{eq1.4}
that after metric identifications $in_{\tilde r'}$ passes to an isometric embedding
$em':\Omega_{p,\tilde r}^{X}~\rightarrow~\Omega_{p,\tilde r'}^{X}$ under which the
diagram
\begin{equation} \label{eq1.5}
\begin{array}{ccc}
\tilde X_{p, \tilde r} & \xrightarrow{\ \ \mbox{\emph{in}}_{\tilde r'}\ \ } &
\tilde X_{p, \tilde r^{\prime}} \\
\!\! \!\! \!\! \!\! \! \pi\Bigg\downarrow &  & \! \!\Bigg\downarrow \pi^{\prime}
\\
\Omega_{p, \tilde r}^{X} & \xrightarrow{\ \ \mbox{\emph{em}}'\ \ \ } & \Omega_{p, \tilde
r^{\prime}}^{X}
\end{array}
\end{equation}is commutative. Here $\pi$ and
$\pi'$ are the natural projections, $$\pi(\tilde x):=\{\tilde y \in \tilde X_{p,\tilde
r}: \tilde d_{\tilde r}(\tilde x, \tilde y)=0\} \quad \mbox{and} \quad \pi'(\tilde x):=\{\tilde y \in
\tilde X_{p,\tilde r'}: \tilde d_{\tilde r'}(\tilde x, \tilde y)=0\}.$$

Let $X$ and $Y$ be metric spaces. Recall that a map $f:X\rightarrow Y$ is called an
\emph{isometry} if $f$ is distance-preserving and onto.

\begin{definition}\label{D1.4}A pretangent $\Omega_{p,\tilde
r}^{X}$ is tangent if $em':\Omega_{p,\tilde r}^{X}\rightarrow \Omega_{p,\tilde r'}^{X}$
is an isometry for every~$\tilde r'.$\end{definition}

The following lemmas will be used in sections 3 and 4 of the paper.

\begin{lemma}\label{Lem1.6}\emph{\cite{DAK}}
Let $\mathbf{\mathfrak{B}}$ be a countable subfamily
of $\tilde X$ and let $\tilde\rho=(\rho_n)_{n\in\mathbb N}$ be a scaling sequence.
Suppose that the inequality $$\limsup_{n\to\infty}\frac{d(b_n,p)}{\rho_n}<\infty$$ holds for every $\tilde b=(b_n)_{n\in\mathbb N}\in\mathbf{\mathfrak{B}}.$ Then there is an infinite
subsequence $\tilde \rho'$ of $\tilde \rho$ such that the family
$\mathbf{\mathfrak{B'}}=\{\tilde b': \tilde b \in
\mathbf{\mathfrak{B'}}\}$ is self-stable w.r.t. $\tilde \rho'.$
\end{lemma}
The next lemma follows from Corollary 3.3 of \cite{D}.

\begin{lemma}\label{Lem1.7}
Let $(X,d)$ be a metric space and let $Y, Z$ be dense subsets of $X.$ Then for every $p\in Y \cap Z$ and every $\Omega_{p, \tilde r}^{Y}$ there are $\Omega_{p, \tilde r}^{Z}$ and an isometry $f:\Omega_{p, \tilde r}^{Y}\rightarrow \Omega_{p, \tilde r}^{Z}$ such that $f(\alpha_Y)=\alpha_Z$ where $\alpha_Y$ and $\alpha_Z$ are the marked points of $\Omega_{p, \tilde r}^{Y}$ and $\Omega_{p, \tilde r}^{Z}$ respectively.
\end{lemma}

It was proved in \cite{DAK} that a bounded tangent space to $X$ at $p$ exists if and only if the distance set
\begin{equation*}
S_{p}(X)=\{d(x,p): x\in X\}
\end{equation*}
is strongly porous at 0. The necessary and sufficient conditions under which all pretangent spaces to $X$ at $p$ are bounded also formulated in terms of the local porosity of the set $S_{p}(X)$ (see \cite{BD1} for details).

In the present paper we shall consider some interconnections between pretangent spaces and a subclass of the locally strongly porous on the right sets.

%%%%%%%%%%%%%%%%%%%%%%%%%%%Section 2%%%%%%%%%%%%%%%%%%%%%%%%%

\section {Completely strongly porous sets}
\hspace*{\parindent} Let us recall the definition of the right hand porosity. Let $E$ be a subset of $\mathbb R^{+}=[0,\infty).$
\begin{definition}\emph{\cite{Th}}\label{D1}
The right hand porosity of $E$ at 0 is the quantity
\begin{equation}\label{L1}
p^{+}(E,0):=\limsup_{h\to 0^{+}}\frac{\lambda(E,0,h)}{h}
\end{equation}
where $\lambda(E,0,h)$ is the length of the largest open subinterval of $(0,h)$ that
contains no points of $E$. The set $E$ is strongly porous at 0 if $p^{+}(E,0)=1.$
\end{definition}
Let $\tilde \tau=(\tau_n)_{n\in\mathbb N}$ be a sequence of real numbers. We shall say
that $\tilde \tau$ is almost decreasing if the inequality $\tau_{n+1}\le\tau_{n}$ holds
for sufficiently large $n.$ Write $\tilde E_{0}^{d}$ for the set of almost decreasing
sequences $\tilde \tau$ with $\mathop{\lim}\limits_{n\to\infty}\tau_{n}=0$ and
$\tau_{n}\in E\setminus \{0\}$ for $n\in\mathbb N.$

Define $\tilde I_{E}^{d}$ to be the set of sequences of open intervals $(a_n,b_n)\subseteq
\mathbb R^{+}, n\in\mathbb N,$ meeting the conditions:

\bigskip
$\bullet$ \emph{Each $(a_n, b_n)$ is a connected component of the set $Ext E=Int(\mathbb
R^{+}\setminus E),$ i.e., $(a_n,b_n)\cap E=\varnothing$ but  $$((a,b)\ne (a_n, b_n))\Rightarrow((a,b)\cap E \ne
\varnothing)$$} \emph{for every}
$(a,b)\supseteq(a_n,b_n);$

$\bullet$ $(a_n)_{n\in\mathbb N}$ \emph{is almost decreasing};

\emph{$\bullet$
$\mathop{\lim}\limits_{n\to\infty}a_{n}=0$ and
$\mathop{\lim}\limits_{n\to\infty}\frac{b_n-a_n}{b_n}=1.$}

\bigskip

Define also an equivalence $\asymp$ on the set of sequences of strictly positive
numbers as follows. Let $\tilde a=(a_n)_{n\in\mathbb N}$ and
$\tilde{\gamma}=(\gamma_n)_{n\in\mathbb N}.$ Then $\tilde a \asymp \tilde {\gamma}$ if
there are some constants $c_1, c_2
>0$ such that
$
c_1 a_n < \gamma_n < c_2 a_n
$
for $n\in\mathbb N.$
\begin{definition}\label{D2*}
Let $E\subseteq\mathbb R^{+}$ and let $\tilde \tau
\in \tilde E_{0}^{d}.$ The set $E$ is $\tilde \tau$-strongly porous at 0 if there is a
sequence $\{(a_n, b_n)\}_{n\in\mathbb N}\in\tilde I_{E}^{d}$ such that
$
\tilde\gamma \asymp \tilde a
$
where $\tilde a=(a_n)_{n\in\mathbb N}.$ $E$ is completely strongly porous
at 0 if $E$ is $\tilde \tau$-strongly porous at 0 for every $\tilde \tau \in \tilde
E_{0}^{d}.$
\end{definition}

The last definition is an equivalent  form of Definition 2.4 from \cite{BD2} (that follows directly from Lemma 2.11 in \cite{BD2}). We denote by \textbf{\emph{CSP}} the set of all completely strongly porous at 0 subsets of $\mathbb R^{+}$. It is clear that every $E\in\textbf{\emph{CSP}}$ is strongly porous at 0 but not conversely. Moreover, if 0 is an isolated point of $E\subseteq\mathbb R^{+},$ then $E\in \textbf{\emph{CSP}}.$

The next lemma immediately follows from Definition~\ref{D2*}.

\begin{lemma}\label{Lem2.3}
Let $E\subseteq\mathbb R^{+},$
$\tilde\gamma\in\tilde E_{0}^{d},$  $\{(a_n, b_n)\}_{n\in\mathbb N}\in\tilde I_{E}^{d}$
and let $\tilde a=(a_n)_{n\in\mathbb N}.$ The equivalence
$\tilde\gamma\asymp\tilde a$ holds if and only if we have
\begin{equation*}\label{L2*}
\limsup_{n\to\infty}\frac{a_n}{\gamma_n}<\infty \quad\mbox{and}\quad \gamma_n \le a_n
\end{equation*} for sufficiently large $n.$
\end{lemma}

Define a set $\mathbb N_{N_1}$ as $\{N_1, N_1+1,...\}$ for $N_1\in\mathbb N.$

\begin{definition}\label{univ} Let $$\tilde A =\{(a_n,b_n)\}_{n\in\mathbb
N}~\in~\tilde I_{E}^{d}\quad\mbox{and}\quad \tilde L =\{(l_n,m_n)\}_{n\in\mathbb
N}~\in~\tilde I_{E}^{d}.$$ Write $\tilde A\preceq\tilde L$ if there are $N_1 \in\mathbb N$ and $f:~\mathbb N_{N_1}~\rightarrow~\mathbb N$ such that
$ a_n = l_{f(n)}$ for every $n\in\mathbb N_{N_1}.$
$\tilde L$ is universal if $\tilde B \preceq \tilde
L$ holds for every $\tilde B\in\tilde I_{E}^{d}.$
\end{definition}

Let $\tilde L=\{(l_n, m_n)\}_{n\in\mathbb N}\in\tilde I_{E}^{d}$ be universal and let
\begin{equation}\label{L13}
M(\tilde L):=\limsup_{n\to\infty}\frac{l_n}{m_{n+1}}.
\end{equation}

In what follows $ac E$ means the set of all accumulation points of a set $E.$
\begin{theorem}\emph{\cite{BD2}}\label{ImpTh}
Let $E\subseteq\mathbb R^{+}$ be strongly porous at 0 and $0\in ac E.$ Then $E\in \textbf{CSP}$ if and only if there is an universal $\tilde L\in \tilde I_{E}^{d}$ such that $M(\tilde L)<\infty.$
\end{theorem}

Note that the quantity $M(\tilde L)$ depends from the set $E$ only (for details see \cite{BD2}).
The following lemma is used in next part of the paper.

\begin{lemma}\label{Pr5}\emph{\cite{BD2}}
Let $E\subseteq\mathbb R^{+}$ and let
$\tilde \tau=(\tau_n)_{n\in\mathbb N}\in\tilde E_{0}^{d}.$ Then $E$ is $\tilde\tau$-strongly porous at 0 if and only if there is a constant $k\in (1, \infty)$ such that for every $K\in (k, \infty)$ there exists $N_{1}(K)\in \mathbb N$ with
$ (k\tau_n, K\tau_n)\cap E = \varnothing$ for every $n
\ge N_1(K).$
\end{lemma}

Let us consider now a simple set belonging to \emph{\textbf{CSP}}.

\begin{example}\label{ex4.4.7}
Let $(x_n)_{n\in\mathbb N}$ be strictly decreasing sequence of positive real numbers with $\mathop{\lim}\limits_{n\to\infty}\frac{x_{n+1}}{x_n}=0.$ Define a set $W$ as $$W=\{x_n: n\in\mathbb N\},$$ i.e., $W$ is the range of the sequence $(x_n)_{n\in\mathbb N}.$ Then $W\in \textbf{\emph{CSP}}$ and $\tilde L=\{(x_{n+1}, x_n)\}_{n\in\mathbb N}\in\tilde I_{W}^{d}$ is universal with $M(\tilde L)=1.$
\end{example}

\begin{proposition}\label{pr4.4.8}
 Let $E\subseteq\mathbb R^{+}.$ Then the inclusion
\begin{equation}\label{4.4.22}
\{E\cup A: A\in \textbf{CSP}\}\subseteq \textbf{CSP}
\end{equation}
holds if and only if 0 is an isolated point of $E.$
\end{proposition}

\begin{remark}\label{rem4.4.9}
Inclusion \eqref{4.4.22} means that $E\cup A\in \textbf{\emph{CSP}}$ for every $A\in \textbf{\emph{CSP}}.$
\end{remark}

\emph{Proof of Proposition~\ref{pr4.4.8}.} If $0\notin ac E,$ then \eqref{4.4.22} follows almost directly and we omit the details here. Suppose $0\in ac E.$  Then there is a sequence $(\tau_n)_{n\in\mathbb N}$ such that $\tau_n\in E$ and $\tau_{n+1}\le 2^{-n^{2}}\tau_n$ for every $n\in\mathbb N.$ Let $M_{1}, M_2,..., M_{k},...$ be an infinite partition of $\mathbb N,$ $$\mathop{\cup}\limits_{k=1}^{\infty}M_{k}=\mathbb N, \, M_{i}\cap M_{j}=\varnothing \,\mbox{if} \, i\ne j$$ such that card$M_{k}=$card$\mathbb N$ for every $k$ and $\nu(1)<\nu(2)<...<\nu(k)...$ where \begin{equation}\label{4.4.23}\nu(k)=\mathop{\min}\limits_{n\in M_k}n.\end{equation} Let $n\in\mathbb N$ and let $m(n)$ be the index such that $n\in M_{m(n)}.$ For every $n\in\mathbb N$ define $\tau_n^{*}$ as $2^{-m(n)}\tau_n.$ Write
\begin{equation*}
E_{1}=\{\tau_n: n\in\mathbb N\} \quad\mbox{and}\quad E_{1}^{*}=\{\tau_{n}^{*}: n\in\mathbb N\}.
\end{equation*}
Using Lemma~\ref{Pr5} we can show that $E_{1}\cup E_{1}^{*}$ is not $\tau^{*}$-strongly porous with $\tilde\tau^{*}=(\tau_n)_{n\in\mathbb N}.$ Consequently $E_{1}\cup E_{1}^{*} \notin \textbf{\emph{CSP}}.$ It implies that $E \cup E_{1}^{*} \notin \textbf{\emph{CSP}}$ because $E_{1}\subseteq E.$ To complete the proof, it suffices to show that $E_{1}^{*}\in \emph{\textbf{CSP}}.$ To this end, we note that \eqref{4.4.23} and the inequalities $\nu(1)<\nu(2)<...<\nu(k)...$ imply that $m(n)\le n$ for every $n\in\mathbb N.$ Indeed, if $m(n)=k,$ then we have $$n\ge\nu(k)=(\nu(k)-\nu(k-1))+(\nu(k-1)-\nu(k-2))+...+ (\nu(2)-\nu(1))+\nu(1)$$ $$\ge (k-1)+\nu(1)=k=m(n).$$ Consequently $$\tau_{n}^{*}=2^{-m(n)}\tau_n\ge 2^{-n}\tau_n\ge2^{-n^{2}}\tau_n \ge\tau_{n+1}\ge\tau_{n+1}^{*}$$ for every $n\in\mathbb N.$ It follows from that $$\lim_{n\to\infty}\frac{\tau_{n}^{*}}{\tau_{n+1}^{*}}\ge\lim_{n\to\infty}\frac{2^{-n}\tau_n}{2^{-n^{2}}\tau_n }=\lim_{n\to\infty}2^{n^{2}-n}=+\infty.$$
Thus, as in Example~\ref{ex4.4.7}, we have $E_{1}^{*}\in \textbf{\emph{CSP}}.$ $\qquad\qquad\qquad\qquad\qquad\qquad\square$
%%%%%%%%%%%%%%%%%%%%%%%%%%%%%%%%Section3%%%%%%%%%%%%%%%%%%%%%%%%%%%%%%%%%%%%%%%%%%%%%%%%%%%%

\section{Uniform boundedness and \textbf{\emph{CSP}}}

\hspace*{\parindent} Let $\mathfrak F=\{(X_{i}, d_{i}): i \in I\}$ be a nonempty family
of  metric spaces. The family $\mathfrak F$ is \emph{uniformly bounded} if there is a constant
$c>0$ such that the inequality
$
\textrm{diam}X_i <c
$
holds for every $i\in I.$ If all metric spaces $(X_i,d_i)\in\mathfrak
{F}$ are pointed with marked points $p_i\in X_i,$ then the uniform boundedness of
$\mathfrak {F}$ can be described by the next way. Define
\begin{equation}\label{L4.1}
\rho^{*}(X_i):=\sup_{x\in X_i}d_{i}(x,p_i) \quad \mbox{and} \quad R^{*}(\mathfrak {F}):=\sup_{i\in I}\rho^{*}(X_i).
\end{equation}
The family $ \mathfrak {F}$ is uniformly bounded if and only if
$
R^{*}(\mathfrak {F})<\infty.
$

\begin{proposition}\label{Pr4.1}
Let $(X,d,p)$ be a pointed metric space and let $\mathbf{\Omega_{p}^{X}}$ be the set of all pretangent spaces to $X$ at $p.$ The following statements are equivalent.
\newline $\mathrm{(i_1)}$ \textit{The family $\mathbf{\Omega_{p}^{X}}$ is uniformly bounded.}
\newline $\mathrm{(i_2)}$ \textit{The point $p$ is an isolated point of $X$.}
\end{proposition}
\begin{proof}The implication $\mathrm{(i_2)} \Rightarrow \mathrm{(i_1)}$ follows directly from the definitions. To prove $\mathrm{(i_1)} \Rightarrow \mathrm{(i_2)}$ suppose that $p\in ac X.$  Let $\tilde x=(x_n)_{n\in\mathbb N}\in\tilde X$ be a sequence of
distinct points of X such that $\mathop{\lim}\limits_{n\to\infty}d(x_{n},p)~=~0.$ For
$t>0$ define the scaling sequence $\tilde r_{t}=(r_{n,t})_{n\in\mathbb N}$ with
$r_{n,t}=\frac{d(x_n,p)}{t}.$ It follows at once from Definition 1.1 that $\tilde x$ and
$\tilde p$ are mutually stable w.r.t $\tilde r_{t}$ and
\begin{equation}\label{L4.4}
\tilde d_{\tilde r_{t}}(\tilde x, \tilde p)=t.
\end{equation}
Let $\tilde X_{p,\tilde r_{t}}$ be a maximal self-stable family meeting the relation
$\tilde x\in\tilde X_{p,\tilde r_{t}}.$ Equality \eqref{L4.4} implies the inequality
$$\textrm{diam} \, \Omega_{p,\tilde r_{t}}^{X}\ge t,$$ where $\Omega_{p,\tilde r_{t}}^{X}=\pi(\tilde X_{p,\tilde
r_{t}}).$ Consequently the family $\mathbf{\Omega_{p}^{X}}$ is not uniformly bounded.
The implication $\mathrm{(i_1)} \Rightarrow \mathrm{(i_2)}$ follows.
\end{proof}

The proposition above shows that the question on the uniform boundedness can be
informative only for some special subfamilies of $\mathbf{\Omega_{p}^{X}}.$ We can
narrow down the family $\mathbf{\Omega_{p}^{X}}$ by the way of consideration some special
scaling sequences.

\begin{definition}\label{D4.2}
Let $(X,d,p)$ be a pointed metric space and let $p\in ac X.$ A
scaling sequence $(r_{n})_{n\in\mathbb N}$ is normal if there is $
(x_{n})_{n\in\mathbb N}\in\tilde X$ such that the sequence $(d(x_n,
p))_{n\in\mathbb N}$ is almost decreasing and
\begin{equation}\label{L4.4*}
\lim_{n\to\infty}\frac{d(x_n,p)}{r_{n}}=1.
\end{equation}
\end{definition}

\begin{proposition}\label{prop}The following properties take place for every pointed metric space $(X,d,p).$
\newline $(i_1)$ If $\Omega_{p,\tilde r}^{X}$ contains at least two distinct points, then there are $c>0$ and a subsequence $(r_{n_k})_{k\in\mathbb N}$ of $\tilde r$ so that the sequence $(cr_{n_k})_{k\in\mathbb N}$ is normal.
\newline $(i_2)$ If $(x_n)_{n\in\mathbb N}\in\tilde X$ and \eqref{L4.4*} holds, then there is an infinite increasing sequence $(n_k)_{k\in\mathbb N}$ so that $(d(x_{n_k}, p))_{k\in\mathbb N}$ is decreasing.
\newline $(i_3)$ If $\tilde r$ is a normal scaling sequence,then there is $(x_n)_{n\in\mathbb N}\in\tilde X$ such that \eqref{L4.4*} holds and $(d(x_n, p))_{n\in\mathbb N}$ is almost decreasing decreasing.
\end{proposition}
\begin{proof}
It is easily verified that $(i_1)$ and $(i_2)$ hold. To verify $(i_3)$ observe that there is $(y_n)_{n\in\mathbb N}\in\tilde X$ which satisfies $d(y_n, p)>0$ for every $n\in\mathbb N$ and \eqref{L4.4*} with $(x_n)_{n\in\mathbb N}=(y_n)_{n\in\mathbb N}.$ Let $m(n)\in\mathbb N$ meet the conditions $m(n)\le n$ and $d(y_{m(n},p)=\mathop{\min}\limits_{1\le i \le n}d(y_i, p).$ Since $\tilde r$ is almost decreasing we have $$\frac{d(y_{m(n)},p)}{r_{m(n)}}\le\frac{d(y_{m(n)},p)}{r_n}\le\frac{d(y_{n},p)}{r_n}.$$ The conditions $\mathop{\lim}\limits_{n\to\infty}y_n=p$ and $d(y_n, p)>0$ for $n\in\mathbb N$ imply that $m(n)\to\infty$ as $n\to\infty.$ Consequently $$1=\lim_{n\to\infty}\frac{d(y_{m(n)},p)}{r_{m(n)}}=\lim_{n\to\infty}\frac{d(y_{m(n)},p)}{r_{n}}=\lim_{n\to\infty}\frac{d(y_{n},p)}{r_{n}}.$$ Thus $(i_3)$ holds with $(x_n)_{n\in\mathbb N}=(y_{m(n)})_{n\in\mathbb N}.$
\end{proof}

Write $\mathbf{\Omega_{p}^{X}(n)}$ for the
set of pretangent spaces $\Omega_{p,\tilde r}^{X}$ with normal scaling sequences.
Under what conditions the family $\mathbf{\Omega_{p}^{X}(n)}$ is uniformly bounded?

\begin{remark}\label{Rem4.3}
Of course, the property of scaling sequence $\tilde r$ to be normal depends on the
underlaying space $(X,d,p)$. Nevertheless for every pointed metric space $(X,d,p)$ a
scaling sequence $\tilde r$ is normal for this space if and only if it
is normal for the space $(S_{p}(X), |\cdot |, 0)$. We shall use this simple fact
below in Proposition~\ref{Pr4.4}.
\end{remark}

In the next proposition we define $\mathbf{\Omega_{0}^{E}(n)}$ to be the set of all pretangent spaces to the
distance set $E=S_p (X)$ at 0 w.r.t. normal scaling sequences.
\begin{proposition}\label{Pr4.4} Let $(X,d,p)$ be a pointed metric space and let $E=S_p (X).$
Then we have
\begin{equation}\label{L4.5}
R^{*}(\mathbf{\Omega_{0}^{E}(n)})=R^{*}(\mathbf{\Omega_{p}^{X}(n)})
\end{equation}
where
$R^{*}(\mathbf{\Omega_{p}^{X}(n)})$ and $R^{*}(\mathbf{\Omega_{0}^{E}(n)})$ are defined by \eqref{L4.1}
with $\mathfrak {F}=\mathbf{\Omega_{p}^{X}(n)}$ and $\mathfrak {F}=\mathbf{\Omega_{0}^{E}(n)}$ respectively.
\end{proposition}
\begin{proof}
If $p\notin ac X$, then the set of normal scaling sequences is
empty. Consequently we have $\mathbf{\Omega_{0}^{E}(n)}=\mathbf{\Omega_{p}^{X}(n)}=\varnothing,$ so we suppose that $p\in ac X.$

 For each normal scaling sequence $\tilde r$ and every $\tilde x \in
\tilde X$ having the finite limit $\mathop{\lim}\limits_{n\to\infty}\frac{d(x_n,p)}{r_n}$ we
can find $(s_n)_{n\in\mathbb N}\in\tilde E$ such that
\begin{equation}\label{**}\lim_{n\to\infty}\frac{d(x_n,p)}{r_n}=\lim_{n\to\infty}\frac{s_n}{r_n}.\end{equation}
Hence the inequality
\begin{equation}\label{L4.7}
\rho(\alpha, \beta)\le R^{*}(\mathbf{\Omega_{0}^{E}(n)})
\end{equation}
holds with $\alpha=\pi(\tilde p)$ for every $\beta\in\Omega_{p,\tilde r}^{X}$ and every
$\Omega_{p,\tilde r}^{X}\in\mathbf{\Omega_{p}^{X}(n)}.$ Taking supremum over all $\Omega_{p,\tilde r}^{X}\in\mathbf{\Omega_{p}^{X}(n)}$ and $\beta\in\Omega_{p,\tilde r}^{X},$ we get
\begin{equation}\label{L4.6} R^{*}(\mathbf{\Omega_{0}^{E}(n)})\ge R^{*}(\mathbf{\Omega_{p}^{X}(n)}).\end{equation} It still remains to prove the inequality
\begin{equation}\label{L4.8}
 R^{*}(\mathbf{\Omega_{0}^{E}(n)})\le R^{*}(\mathbf{\Omega_{p}^{X}(n)}).
\end{equation}
As is easily seen, for every normal scaling $\tilde r$ and every $\tilde s\in
\tilde E$ with $\mathop{\lim}\limits_{n\to\infty}\frac{s_n}{r_n}<\infty,$ there is $\tilde x\in\tilde X$ satisfying \eqref{**}. Now reasoning as in the proof of \eqref{L4.6} we obtain \eqref{L4.8}.
Equality \eqref{L4.5} follows from \eqref{L4.6} and
\eqref{L4.8}.
\end{proof}

\begin{lemma}\label{lem4.5}
Let $E\subseteq\mathbb R^{+}$ and let $0\in ac E.$ If the inequality
\begin{equation}\label{L4.9}
R^{*}(\mathbf{\Omega_{0}^{E}(n)})<\infty
\end{equation} holds, then $E\in \textbf{CSP}$.
\end{lemma}
\begin{proof}
Suppose that \eqref{L4.9} holds but there is $\tilde \tau=(\tau_n)_{n\in\mathbb
N}\in\tilde E_0^{d}$ such that $E$ is not $\tilde\tau$-strongly porous at 0. Then, by
Lemma~\ref{Pr5}, for every $k>1$ there is $K\in(k,\infty)$ such that
$
(k\tau_n, K\tau_n)\cap E\ne\varnothing
$ for all $n$ belonging to an infinite set $A\subseteq\mathbb N.$ Let us put \begin{equation}\label{L4.11}
k=2R^{*}(\mathbf{\Omega_{0}^{E}(n)}).
\end{equation}It simply follows from \eqref{L4.1} and Definition~\ref{D4.2} that $R^{*}(\mathbf{\Omega_{0}^{E}(n)})\ge
1.$ Thus $k\ge 2.$ Consequently we can find $K\in(k, \infty)$ and an infinite set
$A=\{n_1,...,n_j, ...\}\subseteq\mathbb N,$ such that for every $n_j\in A$ there
is $x_j \in E$ satisfying the double inequality
\begin{equation}\label{L4.12}
k<\frac{x_j}{\tau_{n_j}}<K.
\end{equation}Thus the sequence $\left(\frac{x_j}{\tau_{n_j}}\right)_{j\in\mathbb N}$ is bounded.
Hence it contains a convergent subsequence.
Passing to this subsequence we obtain
\begin{equation}\label{L4.13}
\lim_{j\to\infty}\frac{x_j}{\tau_{n_j}}<\infty.
\end{equation}
Now \eqref{L4.11} and \eqref{L4.12} imply
\begin{equation}\label{L4.14}
\lim_{j\to\infty}\frac{x_j}{\tau_{n_j}}=\lim_{j\to\infty}\frac{|0-x_j|}{\tau_{n_j}}\ge
2R^{*}(\mathbf{\Omega_{0}^{E}(n)}).
\end{equation}
The scaling sequence $\tilde r=(r_j)_{j\in\mathbb N}$ with $r_j=\tau_{n_j},$
$j\in\mathbb N,$ is normal. The existence of
finite limit \eqref{L4.13} implies that $\tilde x=(x_j)_{j\in\mathbb N}$ and $\tilde
0$ are mutually stable w.r.t $\tilde r.$ Consequently there is a maximal self-stable
family $\tilde E_{0, \tilde r}$ such that $\tilde x, \tilde 0\in\tilde E_{0, \tilde r}.$ Write $\Omega_{0, \tilde
r}^{E}$ for the metric identification of $\tilde E_{0, \tilde r}$ and $\alpha$ for the
natural projection of $\tilde 0.$ Using \eqref{L4.14} and \eqref{L4.1}, we obtain
$$R^{*}(\mathbf{\Omega_{0}^{E}(n)})\ge\sup_{\gamma\in\Omega_{0,\tilde r}^{E}}\rho(\alpha, \gamma)\ge
2R^{*}(\mathbf{\Omega_{0}^{E}(n)}).$$ The last double inequality is inconsistent because
$1\le R^{*}(\mathbf{\Omega_{0}^{E}(n)})< \infty.$ Thus if \eqref{L4.9} holds, then $E$ is
$\tilde \tau$-strongly porous at 0, as required.
\end{proof}

Let $\tilde \tau\in\tilde E_{0}^{d}.$ Define a subset $\tilde I_{E}^{d} (\tilde \tau)$ of
the set $\tilde I_E^{d}$ by the rule:
$$(\{(a_n ,b_n)\}_{n\in\mathbb N}\in\tilde I_{E}^{d}(\tilde \tau))\Leftrightarrow (\{(a_n ,b_n)\}_{n\in\mathbb N}\in\tilde I_{E}^{d}\, \mbox{and}\, $$ $$\tau_n\le a_n \, \mbox{for sufficiently large}\, n\in\mathbb N).$$
Write
\begin{equation}\label{L9}
C(\tilde \tau):=\inf(\limsup_{n\to\infty}\frac{a_n}{\tau_n})\quad \mbox{and}\quad
C_E:=\sup_{\tilde \tau\in\tilde E_{0}^{d}}C(\tilde\tau)
\end{equation} where the infimum is taken over all $\{(a_n ,b_n)\}_{n\in\mathbb N}\in\tilde I_{E}^{d}(\tilde \tau).$

Using Theorem~\ref{ImpTh} we can prove the following

\begin{proposition}\label{P2}\emph{\cite{BD2}}
Let $E\subseteq\mathbb R^{+}$ and $\tilde\tau\in \tilde E_{0}^{d}.$ The set $E$ is $\tilde\tau$-strongly porous at 0 if and only if $C(\tilde\tau)<\infty.$  The membership $E\in \textbf{CSP}$ holds if and only if $C_E <\infty.$
\end{proposition}

\begin{remark}\label{R1}
If $E\subseteq\mathbb R^{+}, \, i=1,2, \,  \{(a_{n}^{(i)}, b_{n}^{(i)})\}_{n\in\mathbb N}\in\tilde I_{E}^{d}$ and $\tilde a^{1}\asymp \tilde a^{2}$ where $\tilde a^{i}=(a_{n}^{(i)})_{n\in\mathbb N}, \, $ then there is $n_0 \in\mathbb N$ such that $(a_{n}^{(1)}, b_{n}^{(1)})=(a_{n}^{(2)}, b_{n}^{(2)})$ for every $n\ge n_0.$ Consequently if $E$ is $\tilde\tau$-strongly porous and $\{(a_n, b_n)\}_{n\in\mathbb N}\in\tilde I_{E}^{d}(\tilde\tau),$ then we have
\begin{equation*}\label{z1}
\mbox{either}\quad
\limsup_{n\to\infty}\frac{a_n}{\tau_n}=\infty
\quad\mbox{or}\quad \limsup_{n\to\infty}\frac{a_n}{\tau_n}=C(\tilde\tau)<\infty.\end{equation*}
\end{remark}

\begin{lemma}\emph{\cite{BD2}}\label{Lem2.15}
Let $E\in \textbf{CSP}.$ If $\tilde L\in\tilde
I_{E}^{d}$ is universal,
then
$
M(\tilde L)=C_E
$
where the quantities $M(\tilde L)$ and $C_E$ are defined by \eqref{L13} and \eqref{L9}
respectively.
\end{lemma}

\begin{proposition}\label{Pr4.6}
Let $E\subseteq\mathbb R^{+}$ and let $0\in ac E.$ Then the
equality
\begin{equation}\label{L4.15}
C_E =R^{*}(\mathbf{\Omega_{0}^{E}(n)})
\end{equation} holds.
\end{proposition}
\begin{proof}
Let us prove the inequality
\begin{equation}\label{L4.16}
C_E \ge R^{*}(\mathbf{\Omega_{0}^{E}(n)}).
\end{equation}This is trivial if $C_E=\infty.$  Suppose
that
$
C_E <\infty.
$
Inequality \eqref{L4.16} holds if, for every normal scaling sequence $\tilde
r=(r_n)_{n\in\mathbb N}$ and each $\tilde y=(y_n)_{n\in\mathbb N}\in\tilde E_{0}^{d},$
the existence of the finite limit
$\mathop{\lim}\limits_{n\to\infty}\frac{y_n}{r_n}$ implies the
inequality
\begin{equation}\label{L4.18}
\lim_{n\to\infty}\frac{y_n}{r_n}\le C_E.
\end{equation}
Since $\tilde r$ is normal, Proposition~\ref{prop} implies that there is $\tilde x=(x_n)_{n\in\mathbb N}\in\tilde E_{0}^{d}$
with $\mathop{\lim}\limits_{n\to\infty}\frac{r_n}{x_n}=1.$ Consequently \eqref{L4.18} holds if
and only if
\begin{equation}\label{L4.19}
\lim_{n\to\infty}\frac{y_n}{x_n}\le C_E.
\end{equation}
If $\mathop{\lim}\limits_{n\to\infty}\frac{y_n}{x_n}=0,$ then \eqref{L4.19} is trivial.
Suppose that $0<\mathop{\lim}\limits_{n\to\infty}\frac{y_n}{x_n}<\infty.$ The last double
inequality implies the equivalence $\tilde x\asymp\tilde y.$ In accordance with
Proposition~\ref{P2}, $E\in \textbf{\emph{CSP}}$ if and only if $C_{E}<\infty$
holds. Hence $E$ is $\tilde x$-strongly porous at 0. Consequently there is $\{(a_n,
b_n)\}_{n\in\mathbb N}\in\tilde I_{E}^{d}$ such that $\tilde x\asymp\tilde a.$ The relations
$\tilde x\asymp\tilde y$ and $\tilde x\asymp\tilde a$ imply $\tilde y\asymp\tilde a.$
Using Lemma~\ref{Lem2.3} we can find $N_0 \in \mathbb N$ such that $y_n \le a_n$ for
$n\ge N_0.$ Consequently we have $\frac{y_n}{x_n}\le\frac{a_n}{x_n}$ for $n\ge N_0,$
 which implies
$$\lim_{n\to\infty}\frac{y_n}{x_n}\le\limsup_{n\to\infty}\frac{a_n}{x_n}\le C(\tilde\tau)\le
C_E$$ (see \eqref{L9}). Inequality \eqref{L4.16} follows.

To prove \eqref{L4.15}, it
still remains to verify the inequality
\begin{equation}\label{L4.20}
C_E\le R^{*}(\mathbf{\Omega_{0}^{E}(n)}).
\end{equation}
It is trivial if $R^{*}(\mathbf{\Omega_{0}^{E}(n)})=\infty.$ Suppose that
\begin{equation}\label{L4.21}
 R^{*}(\mathbf{\Omega_{0}^{E}(n)})<\infty.
\end{equation}
Inequality \eqref{L4.20}
holds if
\begin{equation}\label{L4.22}
 C(\tilde x)\le R^{*}(\mathbf{\Omega_{0}^{E}(n)}).
\end{equation} for every $\tilde x\in\tilde E_{0}^{d}.$ Let $E^{1}$ denote the closure of the set $E$ in $\mathbb R^{+}$ and let $\tilde x\in\tilde E_{0}^{d}.$ It follows at once from Lemma~\ref{Lem1.7}
that
$R^{*}(\mathbf{\Omega_{0}^{E}(n)})=R^{*}(\mathbf{\Omega_{0}^{E^{1}}(n)}).$ Consequently
\eqref{L4.22} holds if $C(\tilde x)\le R^{*}(\mathbf{\Omega_{0}^{E^{1}}(n)}).$ By
Lemma~\ref{lem4.5}, inequality \eqref{L4.21} implies that $E\in \textbf{\emph{CSP}}$. Hence, by Lemma~\ref{Lem2.3}, there is $\{(a_n, b_n)\}_{n\in\mathbb
N}\in\tilde I_{E}^{d}$ such that \begin{equation}\label{L4.23}
 \limsup_{n\to\infty}\frac{a_n}{x_n}<\infty
\end{equation} and $a_n \ge x_n$ for sufficiently large $n.$ Inequality \eqref{L4.23}
implies the equality
\begin{equation}\label{L4.24}
 C(\tilde x)=\limsup_{n\to\infty}\frac{a_n}{x_n},
\end{equation} (see Remark~\ref{R1}). Let $(n_j)_{j\in\mathbb N}$ be an infinite increasing sequence for which
\begin{equation}\label{L4.25}
 \lim_{j\to\infty}\frac{a_{n_j}}{x_{n_j}}=\limsup_{n\to\infty}\frac{a_n}{x_n}.
\end{equation} Define $r_j :=x_{n_j},$ $\tilde r:=(r_j)_{j\in\mathbb N}$ and
$t_j:=a_{n_j},$ $\tilde t:=(t_j)_{j\in\mathbb N}.$ It is clear that $\tilde r$ is a
normal scaling sequence. Relation \eqref{L4.23} and \eqref{L4.25} imply that $\tilde
t$ and $\tilde 0=(0,0,...,0,...)$ are mutually stable w.r.t. $\tilde r.$ Let $\tilde
E_{0, \tilde r}^{1}$ be a maximal (in $\tilde E^{1}$) self-stable family containing
$\tilde t$ and $\tilde 0.$ Using \eqref{L4.1}, \eqref{L4.24} and \eqref{L4.25}, we
obtain $$R^{*}(\mathbf{\Omega_{0}^{E^{1}}(n)})\ge\sup_{\tilde y\in\tilde E_{0,\tilde
r}^{1}}\tilde d_{\tilde r}(\tilde y, \tilde 0)\ge \tilde d_{\tilde r}(\tilde t, \tilde
0)=C(\tilde x).
$$ Hence \eqref{L4.20} holds that completes the proof of \eqref{L4.15}.
\end{proof}
The following theorem gives the necessary and sufficient conditions under which $\mathbf{\Omega_{p}^{X}(n)}$ is uniformly bounded.
\begin{theorem}\label{th3.10}
Let $(X,d,p)$ be a pointed metric space and let $E=S_{p}(X).$ The
family $\mathbf{\Omega_{p}^{X}(n)}$ is uniformly bounded if and only if
$E\in \textbf{CSP}.$ If $\mathbf{\Omega_{p}^{X}(n)}$ is
uniformly bounded and $p\in ac X$, then
\begin{equation}\label{L4.27}
R^{*}(\mathbf{\Omega_{p}^{X}(n)})=M(\tilde L)
\end{equation}
where $\tilde L$ is an universal element of $(\tilde
I_{E}^{d}, \preceq)$ and $M(\tilde L)$ is defined by \eqref{L13}.
\end{theorem}
\begin{proof} The theorem is trivial if $p$ is an isolated point of $X,$ so that we assume $p\in ac X.$ By Proposition~\ref{Pr4.4}, $\mathbf{\Omega_{p}^{X}(n)}$ is uniformly bounded
if and only if $\mathbf{\Omega_{0}^{E}(n)}$ is uniformly bounded. Since $C_E
=R^{*}(\mathbf{\Omega_{0}^{E}(n)})$ (see \eqref{L4.15}), $\mathbf{\Omega_{0}^{E}(n)}$ is
uniformly bounded if and only if $C_{E}<\infty.$ Using Proposition~\ref{P2} we obtain
that $\mathbf{\Omega_{0}^{E}(n)}$ is uniformly bounded if and only if $E\in \textbf{\emph{CSP}}.$

Let us prove that \eqref{L4.27} holds if $\mathbf{\Omega_{p}^{X}(n)}$ is uniformly
bounded and $p\in ac X$. In this case, as was proved above, $E\in \textbf{\emph{CSP}}.$
Consequently, by Theorem~\ref{ImpTh}, there is an universal element $\tilde L
\in\tilde I_{E}^{d}$ such that $M(\tilde L)<\infty.$
Lemma~\ref{Lem2.15} implies that
\begin{equation}\label{L4.29}
M(\tilde L)=C_E.
\end{equation}
By Proposition~\ref{Pr4.6}, we have also the equality
\begin{equation}\label{L4.30}
C_E=R^{*}(\mathbf{\Omega_{0}^{E}(n)}).
\end{equation}
Since $R^{*}(\mathbf{\Omega_{0}^{E}(n)})=R^{*}(\mathbf{\Omega_{p}^{X}(n)}),$ equalities \eqref{L4.29} and
\eqref{L4.30} imply \eqref{L4.27}.
\end{proof}

Theorem~\ref{th3.10} is an example of translation of some results related to completely strongly porous at 0 sets on the language of pretangent spaces. In the rest of the present section of the paper we shall obtain one more example.

Let $(X_i, d_i, p_i), i=1,2,$ be pointed metric spaces. We write $$(Y, d_y, p_y)=X_{1}\uplus X_{2}$$ if $(Y, d_y, p_y)$ is a pointed metric space for which there are $Y_{i}\subseteq Y$ and isometries $f_{i}: Y_{i}\rightarrow X_i$ such that $Y_1 \cup Y=Y, \, p_y\in Y_1 \cap Y_2$ and $f_{i}(p_y)=p_i$ with $i=1,2.$ For example, to construct $X_{1}\uplus X_{2}$ we can use some isometric embeddings $F_i$ of $X_i, i=1,2,$ into a linear normed space $X$ equipped with $l_{\infty}$ norm \cite[p. 13]{Se}. The translation $\Phi(x)=x+F_{1}(p_1)-F_{2}(p_2), x\in X,$ is a self-isometry of $X$ transforming $F_{2}(p_2)$ at $F_{1}(p_1).$ Consequently the set $F_{1}(X_1)\cup \Phi(F_2 (X_2))$ with the metric from $X$ and the marked point $F_{1}(p_1)=\Phi(F_{2}(p_2))$ meets all desired properties of $X_{1}\uplus X_{2}.$

Proposition~\ref{pr4.4.8} and Theorem~\ref{th3.10} give us the following
\begin{corollary}\label{col4.4.10}
Let $(X, d_x, p_x)$ be a pointed metric space. The following statements are equivalent:
\item[\rm(i)]\textit{$\mathbf{\Omega_{p_{z}}^{Z}(n)}$ is uniformly bounded for every $(Z,d_z,p_z)=Y\uplus X$ having the uniformly bounded $\mathbf{\Omega_{p_{y}}^{Y}(n)}$;}
\item[\rm(ii)]\textit{$p_x\notin ac X.$ }
\end{corollary}

For the proof it suffices to note that $S_{p_{z}}(Z)=S_{p_{x}}(X)\cup S_{p_{y}}(Y).$

%%%%%%%%%%%%%%%%%%%%%%%%%%%%%%%%%%%%%SECTION4%%%%%%%%%%%%%%%%%%%%%%%%%%%%%%%5

\section {Uniform boundedness and uniform discreteness}
\hspace*{\parindent} Let $\mathfrak F=\{(X_{i}, d_{i}, p_i): i \in I\}$ be a nonempty family
of pointed metric spaces.  We set
\begin{equation}\label{4.4.1}
\rho_{*}(X_i):=\begin{cases}
         \inf\{d_i(x,p_i): x\in X_{i}\setminus\{p_i\}\} & \mbox{if} $ $ X_i \ne \{p_i\}\\
         +\infty & \mbox{if}$ $ X_i = \{p_i\}\\
         \end{cases}
\end{equation}
for $i\in I$ and write
$
R_{*}(\mathfrak F):=\mathop{\inf}\limits_{i\in I}\rho_{*}(X_i).
$
We shall say that $\mathfrak F$ is \emph{uniformly discrete} (w.r.t. the marked points $p_i$) if $R_{*}(\mathfrak F)>0.$

As in Proposition~\ref{Pr4.1} it is easy to show that the family $\mathbf {\Omega_{p}^{X}}$ of all pretangent spaces is uniformly discrete if and only if $p$ is an isolated point of the metric space $X.$ Thus, it make sense to consider $\mathbf {\Omega_{p}^{X}(n)}.$

\begin{theorem}\label{th4.4.1}
Let $(X,d,p)$ be a pointed metric space such that $\mathbf {\Omega_{p}^{X}(n)}\ne\varnothing.$ Then $\mathbf {\Omega_{p}^{X}(n)}$ is uniformly discrete if and only if it is uniformly bounded. The equality
\begin{equation}\label{4.4.3}
R_{*}(\mathbf {\Omega_{p}^{X}(n)})=\frac{1}{R^{*}(\mathbf {\Omega_{p}^{X}(n)})}
\end{equation} holds if $\mathbf {\Omega_{p}^{X}(n)}$ is uniformly bounded.
\end{theorem}

\begin{remark}\label{r4.4.2}
The condition $\mathbf {\Omega_{p}^{X}(n)}\ne\varnothing$ implies that $R^{*}(\mathbf {\Omega_{p}^{X}(n)})>0.$ Putting $\frac{1}{\infty}=0$ we may also expend \eqref{4.4.3} on arbitrary nonvoid families $\mathbf {\Omega_{p}^{X}(n)}.$
\end{remark}

We will prove Theorem~\ref{th4.4.1} in some more general setting.

Let $(X,d,p)$ be a pointed metric space and let $t>0.$ Write $t X$ for the pointed metric space $(X, t d, p)$, i.e. $t X$ is the pointed metric space with the same underluing set $X$ and the marked point $p,$ but equipped with the new metric $t d$ instead of $d$.

\begin{definition}\label{D4.4.2}
Let $\mathfrak F$ be a nonempty family of pointed metric spaces. $\mathfrak F$ is weakly self-similar if for every $(Y,d,p)\in\mathfrak F$ and every nonzero $t \in S_{p}(Y)$ the space $\frac{1}{t}Y$ belongs to $\mathfrak F.$
\end{definition}

\begin{theorem}\label{th4.4.3}
Let $\mathfrak F=\{(X_i, d_i, p_i): i\in I\}$ be a weakly self-similar family of pointed metric spaces. Suppose that the sphere
\begin{equation*}
S_{i}=\{x\in X_i: d_{i}(x, p_i)=1\}
\end{equation*}
is nonvoid for every $i\in I.$ Then $\mathfrak F$ is uniformly bounded if and only if $\mathfrak F$ is uniformly discrete w.r.t. the marked points $p_i, i\in I.$ If $\mathfrak F$ is uniformly bounded, then the equality
\begin{equation}\label{4.4.4}
R_{*}(\mathfrak F)=\frac{1}{R^{*}(\mathfrak F)}
\end{equation}
holds.
\end{theorem}

\begin{proof}
Assume that $\mathfrak F$ uniformly bounded but not uniformly discrete. Then there is a sequence $(x_{i_k})_{k\in\mathbb N}$ such that

\begin{equation*}\label{4.4.5}
x_{i_k}\in X_{i_k}, x_{i_k}\ne p_{i_k} \,\, \mbox{and} \, \, \lim_{k\to\infty}d_{i_k}(x_{i_k}, p_{i_k})=0
\end{equation*}
Since $S_{i}\ne\varnothing$ for every $i\in I,$ we can find a sequence $(y_{i_k})_{k\in\mathbb N}$ for which $y_{i_k}\in X_{i_k}$ and
$
d_{i_k}(y_{i_k}, p_{i_k})= 1
$
for every $k\in\mathbb N.$  Define $t_k$ to be $d_{i_k}(x_{i_k}, p_{i_k}), k\in\mathbb N.$ Since $\mathfrak F$ is weakly self-similar, the membership
$
t_{k}^{-1}X_{i_k}\in\mathfrak F
$
holds for every $k\in\mathbb N.$ Now we obtain

\begin{equation*}
R^{*}(\mathfrak F)\ge\limsup_{k\to\infty}t_{k}^{-1}d_{i_k}(y_{i_k}, p_{i_k})=\limsup_{k\to\infty}t_{k}^{-1}=\infty.
\end{equation*}
Hence $\mathfrak F$ is not uniformly bounded, contrary to the assumption. Therefore if $\mathfrak F$ is uniformly bounded, then $\mathfrak F$ is uniformly discrete. Similarly we can prove that the uniform discreteness of $\mathfrak F$ implies the uniform boundedness of this family.

Suppose now that
$
R^{*}(\mathfrak F)<\infty.
$
Let us prove equality \eqref{4.4.4}. Define a quantity $Q(\mathfrak F)$ by the rule

\begin{equation*}
Q(\mathfrak F)=\sup_{i\in I}\frac{\rho^{*}(X_i)}{\rho_{*}(X_i)}
\end{equation*}
where $\rho_{*}(X_i)$ is defined by \eqref{4.4.1} and $\rho^{*}(X_i)$ by \eqref{L4.1}.The first part of the theorem implies that $\mathfrak F$ is uniformly discrete. Hence $R_{*}(\mathfrak F)>0,$ that implies $\rho_{*}(X_i)>0, i\in I.$ Moreover the inequality $R^{*}(\mathfrak F)<\infty$ gives us the condition $\rho_{*}(X_i)<\infty.$ Thus $Q(\mathfrak F)$ is correctly defined.
We claim that the equality
$
Q(\mathfrak F)=R^{*}(\mathfrak F)
$
holds. Indeed let $(i_k)_{k\in\mathbb N}$ be a sequence of indexes $i_k\in I$ such that

\begin{equation}\label{4.4.10}
\lim_{k\to\infty}\rho_{*}(X_{i_k})=R^{*}(\mathfrak F).
\end{equation}
Since $S_{i}\ne\varnothing$ for every $i\in I,$ we have $\rho_{*}(X_{i_k})\le 1$ for every $X_{i_k}.$ Consequently
\begin{equation}\label{4.4.11}
Q(\mathfrak F)\ge\limsup_{k\to\infty}\frac{\rho^{*}(X_{i_k})}{\rho_{*}(X_{i_k})}\ge\limsup_{k\to\infty}\rho^{*}(X_{i_k})=\lim_{k\to\infty}\rho^{*}(X_{i_k})\ge R^{*}(\mathfrak F).
\end{equation}
Let us consider a sequence $(i_m)_{m\in\mathbb N}, i_m\in I,$ for which
\begin{equation}\label{4.4.12}
Q(\mathfrak F)=\lim_{m\to\infty}\frac{\rho^{*}(X_{i_m})}{\rho_{*}(X_{i_m})}.
\end{equation}
The quantity $\frac{\rho^{*}(X_{i})}{\rho_{*}(X_{i})}$ is invariant w.r.t. the passage from $X_i$ to $\frac{1}{t}X_{i}, t\in S_{p_i}(X_i).$ Consequently using the uniform discreteness of $\mathfrak F$ and the inequality $\rho_{*}(X_i)\le 1$ (which follows from the condition $S_{i}\ne\varnothing, \, i\in I$), we may assume that

\begin{equation}\label{4.4.13}
\lim_{m\to\infty}\rho_{*}(X_{i_m})=1.
\end{equation}
Limit relations \eqref{4.4.12} and \eqref{4.4.13} imply

\begin{equation*}
Q(\mathfrak F)=\lim_{m\to\infty}\rho_{*}(X_{i_m})\le R^{*}(\mathfrak F).
\end{equation*}
The last inequality and \eqref{4.4.11} give us the equality
$R^{*}(\mathfrak F)=Q(\mathfrak F).$
Reasoning similarly we obtain the equality
$
Q(\mathfrak F)=\frac{1}{R_{*}(\mathfrak F)}.
$
Equality \eqref{4.4.4} follows.
\end{proof}

Let us define a subset $^{\textbf{1}}\mathbf{\Omega_{p}^{X}}$ of the set $\mathbf{\Omega_{p}^{X}}$ of all pretangent spaces to $X$ at $p$ by the rule:
\begin{equation*}\label{4.4.16}
\left(\Omega_{p,\tilde r}^{X}, \rho, \alpha\right)\in \,^{\textbf{1}}\mathbf{\Omega_{p}^{X}}\quad\mbox{if and only if $\tilde r$ is almost decreasing}
\end{equation*}
\begin{equation*}
\mbox{and} \quad \{\delta\in\Omega_{p,\tilde r}^{X}: \rho(\alpha,\delta)=1\}\ne\varnothing
\end{equation*}
where $\alpha=\pi(\tilde p)$ is the marked point of the pretangent space $\Omega_{p,\tilde r}^{X}$ and $\rho$ is the metric on $\Omega_{p,\tilde r}^{X}.$  It is clear that $^{\textbf{1}}\mathbf{\Omega_{p}^{X}}$ meets the condition of Theorem~\ref{th4.4.3} and
\begin{equation}\label{4.4.17}
^{\textbf{1}}\mathbf{\Omega_{p}^{X}}\subseteq\mathbf{\Omega_{p}^{X}(n)}.
\end{equation}
To apply Theorem~\ref{th4.4.3} to Theorem~\ref{th4.4.1} we need the following
\begin{lemma}\label{l4.4.4}
Let $\Omega_{p,\tilde r}^{X}\in\mathbf{\Omega_{p}^{X}(n)}.$ Then there is $^{1}\Omega_{p,\tilde\mu}^{X}\in\,^{\textbf{\emph{1}}}\mathbf{\Omega_{p}^{X}}$ such that
\begin{equation}\label{4.4.18}
\rho_{*}(^{1}\Omega_{p,\tilde\mu}^{X})\le\rho_{*}(\Omega_{p,\tilde r}^{X})\le\rho^{*}(\Omega_{p,\tilde r}^{X})\le\rho^{*}(^{1}\Omega_{p,\tilde \mu}^{X}).
\end{equation}
\end{lemma}

\begin{proof}
Let $\tilde X_{p,\tilde r}$ be the maximal self-stable family which metric identification coincides with $\Omega_{p,\tilde r}^{X}.$ We can find some sequences $\tilde a^{i}=(a_n^{i})_{n\in\mathbb N}$ and $\tilde b^{i}=(b_n^{i})_{n\in\mathbb N}, i\in\mathbb N$ such that

\begin{equation}\label{4.4.20}
\lim_{i\to\infty}\rho(\pi(\tilde b^{i}), \alpha)=\rho_{*}(\Omega_{p,\tilde r}^{X})\quad\mbox{and}\quad \lim_{i\to\infty}\rho(\pi(\tilde a^{i}), \alpha)=\rho^{*}(\Omega_{p,\tilde r}^{X}).
\end{equation}
Since $\Omega_{p,\tilde r}^{X}\in\mathbf{\Omega_{p}^{X}(n)},$ the scaling sequence $\tilde r$
is normal. Consequently there is $\tilde c=(c_n)_{n\in\mathbb N}\in\tilde X$ such that $$\lim_{n\to\infty}\frac{d(c_n,p)}{r_n}=1.$$ Let us define a countable family $\mathfrak B\subseteq\tilde X$ as $$\mathfrak B=\{\tilde b^{i}: i\in\mathbb N\}\cup\{\tilde a^{i}: i\in\mathbb N\}\cup\{\tilde c\}.$$ The family $\mathfrak B$ satisfies the condition of Lemma~\ref{Lem1.6}. Consequently there is an infinite subsequence $\tilde r'=(r_{n_k})_{k\in\mathbb N}$ the scaling sequence $\tilde r$ for which the family $$\mathfrak B'=\{(b_{n_k}^{i})_{k\in\mathbb N}: i\in\mathbb N\}\cup\{(a_{n_k}^{i})_{k\in\mathbb N}: i\in\mathbb N\}\cup\{(c_{n_k})_{k\in\mathbb N}\}$$ is self-stable w.r.t. $\tilde r'$. Completing $\mathfrak B'$ to a maximal self-stable family and passing to the metric identification of it we obtain the desired pretangent space $^{1}\Omega_{p,\tilde\mu}^{X}$ with $\tilde\mu=\tilde r'.$
\end{proof}
\emph{Proof of Theorem~\ref{th4.4.1}.} Inclusion \eqref{4.4.17} implies the inequalities
\begin{equation*}
R^{*}(^{\textbf{1}}\mathbf{\Omega_{p}^{X}})\le R^{*}(\mathbf{\Omega_{p}^{X}(n)})
\quad \mbox{and}\quad
R_{*}(\mathbf{\Omega_{p}^{X}(n)})\le R_{*}(^{\textbf{1}}\mathbf{\Omega_{p}^{X}}).
\end{equation*}
The converse inequalities follow from \eqref{4.4.18}. Consequently we have the equalities
\begin{equation}\label{4.4.21}
R^{*}(^{\textbf{1}}\mathbf{\Omega_{p}^{X}})= R^{*}(\mathbf{\Omega_{p}^{X}(n)})\quad\mbox{and}\quad
R_{*}(\mathbf{\Omega_{p}^{X}(n)})=R_{*}(^{\textbf{1}}\mathbf{\Omega_{p}^{X}}).
\end{equation}
Now Theorem~\ref{th4.4.1} follows directly from Theorem~\ref{th4.4.3} and \eqref{4.4.21}.$\qquad\qquad\quad\square$

\bigskip

The sum of theorems~\ref{th4.4.1} and \ref{th3.10} yields the following result that was announced in \cite{BD3}.

\begin{theorem}\label{th4.4.5}
Let $(X,d,p)$ be a metric space with a marked point $p\in ac X.$ Then the following three conditions are equivalent.
\item[\rm(i)]\textit{$\mathbf{\Omega_{p}^{X}(n)}$ is uniformly bounded.}
\item[\rm(ii)]\textit{$\mathbf{\Omega_{p}^{X}(n)}$ is uniformly discrete.}
\item[\rm(iii)]\textit{$S_{p}(X)\in \textbf{CSP}.$ }

Moreover if $\mathbf{\Omega_{p}^{X}(n)}$ is uniformly bounded, then
\begin{equation*}
R^{*}(\mathbf{\Omega_{p}^{X}(n)})=M(\tilde L)\quad\mbox{and}\quad
R_{*}(\mathbf{\Omega_{p}^{X}(n)})=\frac{1}{M(\tilde L)}
\end{equation*} where the quantity $M(\tilde L)$ was defined by \eqref{L13}.
\end{theorem}

%%%%%%%%%%%%%%%%%%%%%%%%%%%%%%%bibliography%%%%%%%%%%%%%%%%%%%%%%%%%%%%%%%%%%%%%

\medskip

\textbf{Viktoriia Viktorivna Bilet}

Institute of Applied Mathematics and Mechanics of NASU, R. Luxemburg str. 74, Donetsk 83114, Ukraine

\textbf{E-mail:} biletvictoriya@mail.ru

\bigskip

\textbf{Oleksiy Al'fredovich Dovgoshey}

Institute of Applied Mathematics and Mechanics of NASU, R. Luxemburg str. 74, Donetsk 83114, Ukraine

\textbf{E-mail:} aleksdov@mail.ru


\begin{thebibliography}{99}

\bibitem{DAK} {\it F. Abdullayev, O. Dovgoshey,  M.
K\"u\c{c}\"uaslan,} Compactness and boundedness of tangent spaces to metric spaces. ---
Beitr. Algebra Geom. {\bf 51} (2010), No. 2, 547--576.

\bibitem{BD1} {\it  V. Bilet, O. Dovgoshey,} Boundedness of pretangent spaces to general metric spaces. ---  arXiv: 1210.0656 [math. MG].

\bibitem{BD3} {\it  V. Bilet, O. Dovgoshey,} Infinitesimal boundedness of metric spaces and strong one-side porosity. --- Reports of the National Academy of Sciences of Ukraine,  (2013), No. 2, 13--18 (in Russian).

\bibitem{BD2} {\it  O. Dovgoshey, V. Bilet,} A kind of local strong one-side porosity. ---
 arXiv: 1205.2335 [math. CA].

\bibitem{DM}{\it  O. Dovgoshey, O. Martio, } Tangent spaces to metric spaces. --- Reports in Math. Helsinki Univ. {\bf
480} (2008), 20 p.

\bibitem{MD} {\it  O. Dovgoshey, O. Martio,} Tangent spaces to general metric spaces. ---
Rev. Roumaine Math. Pures. Appl. {\bf 56} (2011), No. 2, 137--155.

\bibitem{D}{\it O. Dovgoshey,} Tangent spaces to metric spaces and to their subspaces. ---
Ukr. Math. Visn. \textbf{5} (2008), No. 4, 470--487.

\bibitem{Kelley} {\it J. L. Kelley,} General topology. D. Van Nostrand Company,
Princeton, (1965).

\bibitem{Se} {\it M. \'{O}. Searc\'{o}id,} Metric Spaces.
Springer-Verlag, London, (2007).

\bibitem{Th} {\it B. S. Thomson,} Real Functions, Lecture Notes in Mathematics. \textbf{1170},
Springer-Verlag, Berlin, Heidelberg, New York, Tokyo, (1985).



%%%%%%%%%%%%%%%%%%%%%%%%%%%%%%%%%%%%%%%%%%%%%Refrrences%%%%%%%%%%%%%%%%%%%%%%%%%%%%%%%%%%%%%%%%%%%%%%%%%

\end{thebibliography}
\end{document}